\documentclass{amsart}

\usepackage{amsfonts}
\usepackage{amsmath}
\usepackage{amsthm}
\usepackage{amssymb}
\usepackage{latexsym}
\usepackage{enumerate}

\newtheorem{theorem}{Theorem}[section]
\newtheorem{lemma}[theorem]{Lemma}
\newtheorem{proposition}[theorem]{Proposition}
\newtheorem{corollary}[theorem]{Corollary}

\newtheorem{question}[theorem]{Question}
\newtheorem{definition}[theorem]{Definition}

\numberwithin{equation}{section}

\begin{document}

\newcommand{\bb}{\mathfrak{b}}
\newcommand{\cc}{\mathfrak{c}}
\newcommand{\N}{\mathbb{N}}
\newcommand{\R}{\mathbb{R}}
\newcommand{\A}{{\mathbb{R}}_+\cup\{0\}}
\newcommand{\forces}{\Vdash}
\newcommand{\LL}{\mathbb{L}}
\newcommand{\K}{\mathbb{K}}
\newcommand{\Q}{\mathbb{Q}}

\title{Independent families in Boolean algebras with some separation properties}

\author{Piotr Koszmider}
\thanks{The first author was partially supported by the National Science Center research grant 2011/01/B/ST1/00657.} 
\email{P.Koszmider@Impan.pl}
\address{Institute of Mathematics, Polish Academy of Sciences,
ul. \'Sniadeckich 8,  00-956 Warszawa, Poland}

\author{Saharon Shelah}
\thanks{}
\email{shelah@math.huji.ac.il}
\address{Department of Mathematics, The Hebrew University of Jerusalem, 90194 Jerusalem, Isreal}
\address{Rutgers University, Piscataway, NJ 08854-8019, USA}

%
%
%
%
\begin{abstract} 
We prove that any Boolean algebra with the subsequential completeness property
contains an independent family of size continuum. This improves a result of Argyros from the 80ties
which asserted the existence of an uncountable independent family. In fact we 
prove it for a bigger class of  Boolean algebras satisfying much weaker properties. It follows that 
the Stone spaces $K_{\mathcal A}$ of all such Boolean algebras $\mathcal A$
contains a copy of the \v Cech-Stone compactification of the integers
$\beta\N$ and the Banach space $C(K_{\mathcal A})$ has $l_\infty$ as a quotient.
Connections with the Grothendieck property in Banach spaces  are
discussed.
\end{abstract}

\maketitle

\markright{}
\section{Independent families}

By an antichain in a Boolean algebra $\mathcal A$
we will mean a pairwise disjoint subset of $\mathcal A$, i.e., a $\mathcal B\subseteq \mathcal A$ such that
$A\wedge B=0$ whenever $A$ and $B$ are two distinct elements of $\mathcal B$.
An independent family in a Boolean algebra $\mathcal A$ is a family $\{A_i: i\in I\}\subseteq \mathcal A$
such that
$$\bigwedge_{i\in F} A_i\wedge \bigwedge_{i\in G} A_i^{-1}\not=0$$
for any two disjoint finite subsets $F, G\subseteq I$. For more information on Boolean algebras see
\cite{handbookbas}. $K_{\mathcal A}$ will stand for the Stone space of $\mathcal A$ and
$C(K)$ for the Banach space of real valued continuous functions on $K$. $[A]$ will stand for
the clopen subset $\{x\in K_{\mathcal A}: A\in x\}$ for any $A$ in a Boolean algebra $\mathcal A$
and $1_X$ for the characteristic function of $X\subseteq K_{\mathcal A}$.
Unexplained notions concerning Banach spaces can be found in \cite{diestel}.
We will consider several separation properties in Boolean algebras:

\begin{definition}\label{sspdefinition}
 A Boolean algebra $\mathcal A$ is said to have a weak subsequential separation property
if and only if given any countably infinite antichain $\{A_n: n\in \N\}\subseteq \mathcal A$ there is $A\in \mathcal A$
such that both of the sets $$\{n\in \N: A_n\leq A\}\ \ \hbox{and}\ \ \{n\in \N: A_n\wedge A=0\}$$
are infinite.
\end{definition}

This is a natural generalization of the following  well known:

\begin{definition}[\cite{haydon}] A Boolean algebra $\mathcal A$ is said to have a subsequential completeness property
if and only if given any countably infinite antichain $\{A_n: n\in \N\}\subseteq \mathcal A$ there is 
and infinite $M\subseteq \N$ such that in $\mathcal A$ there is the supremum of $\{A_n: n\in M\}$.
\end{definition}

One more separation property introduced in \cite{haydonserdica} is:

\begin{definition}\label{hsspdefinition}
 A Boolean algebra $\mathcal A$ is said to have a  subsequential separation property
if and only if given any countably infinite antichain $\{A_n: n\in \N\}\subseteq \mathcal A$ there
is   an $A\in \mathcal A$
such that  the set
  $$\{ n\in \N: A_{2n}\leq A\ \hbox{and}\  A_{2n+1}\wedge A=0\}$$ is infinite
\end{definition}

The subsequential completeness property was introduced in \cite{haydon}. This paper included applications of 
the subsequential completeness property in the theory of
Banach spaces  as well as a result of Argyros (Proposition 1G)
that every Boolean algebra with the subsequential completeness property has an uncountable independent family.
Here we strengthen this result and prove:

\begin{theorem}\label{maintheorem} Suppose that $\mathcal A$ is an infinite Boolean algebra which has a weak subsequential
separation property. Then $\mathcal A$ contains an independent family of cardinality continuum.
\end{theorem}
\begin{proof} Clearly, the Stone space of $\mathcal A$ cannot have a nontrivial convergent sequences
and so cannot be a dispersed space, consequently $\mathcal A$ is not superatomic and so contains a countable free 
infinite Boolean algebra $\mathcal B\subseteq \mathcal A$. 
$\{0,1\}^{<\N}$ will stand for finite $0-1$ sequences.

A countable free Boolean algebra contains all countable Boolean algebras as subalgebras. We will identify
a particular subalgebra $\mathcal C\subseteq \mathcal A$ generated by generators
$\{A_s: s\in \{0,1\}^{<\N}\}$ such that whenever  $\mathcal D$ is a countable free Boolean algebra
generated by free generators $\{D_s: s\in \{0,1\}^{<\N}\}$, then there is a surjective homomorphism
$h: \mathcal D\rightarrow \mathcal C$ such that $h(D_s)=A_s$ and the kernel of
$h$ is the ideal generated by the elements of the form $A_s\wedge A_t$ such that $s\subsetneq t$.

In particular it follows that whenever $s\subseteq t$, then
$$A_s\wedge A_t=0.\leqno *)$$
Now we will note that if $n$ is an integer, $u_i\in\{0,1\}^{<\N}$ for $1\leq i\leq n$, then
$$A_{u_1}\wedge... \wedge A_{u_n}\not=0\leqno **)$$ 
unless $u_i\subseteq u_j$ for some distinct $1\leq i,j\leq n$.
 Indeed, suppose that $u_i\not\subseteq u_j$ for
all distinct  $1\leq i, j\leq n$, consider an arbitrary element of the kernel of $h$, it may be assumed to be of the form
$$\bigvee_{1\leq i\leq k} (D_{s_i}\wedge D_{t_i})$$
for some $k\in \N$ and $s_i, t_i\in\{0,1\}^{<\N}$ satisfying $s_i\subsetneq t_i$ for all $1\leq i\leq k$.
Note that $\{u_1, ..., u_n\}\cap \{s_j, t_j\}$ has at most one element for
each $1\leq j\leq k$. Let $v_j\in \{s_j, t_j\}$ be the other element for $1\leq j\leq k$. 
So $\{v_1, ..., v_k\}$ and $\{u_1, ..., u_n\}$ are disjoint and by the independence of the generators
we have
$$0\not=\bigwedge_{1\leq i\leq k} D_{v_i}^{-1}\wedge \bigwedge_{1\leq i\leq n} D_{u_i}$$
so
$$0\not=\bigwedge_{1\leq i\leq k} (D_{s_i}^{-1}\vee D_{t_i}^{-1})\wedge
 \bigwedge_{1\leq i\leq n} D_{u_i}$$
and hence
$$\bigwedge_{1\leq i\leq n} D_{u_i}\not\leq \bigvee_{1\leq i\leq k} (D_{s_i}\wedge D_{t_i})$$
as required for **).
Now by *) for every $x\in 2^\N$  consider an antichain $\{A_s: s\subseteq x\}$. By the
subsequential separation property for each $x\in 2^\N$ there are $A_x\in \mathcal A$ as in \ref{sspdefinition}.
We will show that $\{A_x: x\in 2^\N\}$ is the required independent family. Let $x_1, ..., x_n$ be distinct
elements of $2^\N$ and let $F\subseteq \{1,..., n\}$ and $G=\{1,..., n\}\setminus F$. Let
$m\in \N$ be such that $u_i=x_i|m$ are all distinct (and so none is included in the other). 
For $i\leq n$ let $m_i\geq m$ be such that $A_{x_i|m_i}\leq A_{x_i}$ if $i\in F$ and 
$A_{x_i|m_i}\wedge A_{x_i}=0$ if $i \in G$. The sequences $x_1|m_1,..., x_n|m_n$ satisfy
the hypothesis of **) for $u_1,..., u_n$, so 
$$0\not=\bigwedge_{1\leq i\leq n} A_{x_i|m_i}\leq
\bigwedge_{i\in F} A_{x_i}\wedge \bigwedge_{i\in G} A_{x_i}^{-1}$$
as required for the independence of $\{A_x: x\in 2^\N\}$.
\end{proof}

In the literature there are several more weakenings of the subsequential completeness property  which are stronger that 
our weak subsequential separation property, most notably the subsequential interpolation property introduced in 
\cite{freniche}
and applied in several other papers
and the subsequential separation property of \cite{haydonserdica}. Hence the above theorem applies to the algebras with these properties as well.

\begin{corollary} If $\mathcal A$ is a Boolean algebra having a 
subsequential separation property and $K_{\mathcal A}$ is its Stone space, then
$\beta \N$ is a subspace of $K_{\mathcal A}$ and $l_\infty$ is a quotient of
$C(K_{\mathcal A})$.
\end{corollary}
\begin{proof} It is well know that  a free Boolean algebra with continuous many
generators maps homeomorphically onto any Boolean algebra of cardinaliy continuuum, and so
e.g. onto $\wp(\N)$. Use the Sikorski extension theorem to obtain a homeomorphism
of $\mathcal A$ onto $\wp(\N)$. By the Stone duality, it follows that
the Stone space of  $\mathcal A$, which is homeomorphic to $\beta\N$,  is a subspace
of $K_{\mathcal A}$. Restricting continuous functions on $K_{\mathcal A}$ to a copy of
$\beta\N$ gives, by the Tietze extension theorem, a norm one linear operator onto
$C(\beta\N)$ which is known to be isometric to $\l_\infty$.
\end{proof}

It follows that many Banach spaces present in the literature  have $l_\infty$ as a quotient.
In particular the spaces of \cite{haydon} or \cite{few}. In \cite{few} besides a Boolean
separations a lattice version of the subsequential
completeness property is considered for connected compact $K$. Namely we consider (see 5.1. of \cite{few})
spaces $K$ such that given any pairwise disjoint ($f_n^. f_m=0$) sequence $(f_n)_{n\in \N}$
of continuous functions $f_n: K\rightarrow [0, 1]$  there  is an infinite
$M\subseteq \N$ such that in $C(K)$ there is the supremum of $(f_n)_{n\in M}$
It is not difficult to generalize the proof of \ref{maintheorem} to conclude that such $K$s 
 always contain $\beta\N$ as well.

\section{The Grothendieck property of Banach spaces}

In this section we would like to direct the attention of the reader to some links between
the weak subsequential separation property and the Grothendieck property which originated in
the theory of Banach spaces.

\begin{definition} Let $X$ be a Banach space.  We say that $X$ has the Grothendieck property
if and only if the weak$^*$ convergence of sequences in $X^*$ is equivalent to the weak convergence.
A Boolean algebra has the Grothendieck property if and only if the Banach spaces
$C(K_{\mathcal A}$ has the Grothendieck property.
\end{definition}

Grothendieck property for Boolean algebras was introduced and first investigated by Schachermayer in
\cite{schachermayer}.  It can be relatively nicely characterized using finitely additive signed measures
on Boolean algebras. Recall that
the Riesz representation theorem says that the dual to a $C(K)$ space is isometric to the space of Radon
measures on $K$ with the variation norm, i.e., all continuous functionals on $C(K)$ are of the form $\int fd\mu$ for
$\mu$ a Radon measure on $K$. Radon measure means signed, countably additive, borel, regular measure.
If $K$ is totally disconnected (has a basis of clopen sets), then any finitely additive, bounded, signed  measure
on $Clop(\mathcal A)$ extends uniquely to a Radon measure  on $K$. So there is a one to one correspondence
between such measures and elements of the dual Banach space to $C(K_{\mathcal A})$. For
more information on this see \cite{semadeni}.

\begin{lemma}\label{lemmagrothen} Suppose $\mathcal A$ is a Boolean algebra, $K_{\mathcal A}$ its Stone space and
$C(K_{\mathcal A})$ the Banach space of all real-valued continuous functions on $K_{\mathcal A}$
with the supremum norm. $C(K_{\mathcal A})$ has the Grothendieck property if and only if
whenever 
\begin{itemize} 
\item $\{A_n: n\in \N\}$ is an antichain of $\mathcal A$,
\item $\varepsilon>0$,
\item $\mu_n$ is a bounded sequence of bounded, finitely additive signed measures on $\mathcal A$
such that $|\mu_n(A_n)|>\varepsilon$
\end{itemize}
then there is $A\in \mathcal A$ such that 
$$(\mu_n(A): n\in \N)$$
is not a convergent sequence of the reals.
\end{lemma}
\begin{proof}
Suppose that $C(K_{\mathcal A})$  has the Grothendieck property and $A_n$s, $\varepsilon$ and
$\mu_n$s are as above. Let $x_n^*$ be elements of the dual to  $C(K_{\mathcal A})$ which
are uniquely determined by the condition $x_n^*(1_{[A]})=\mu_n(A)$ for $A\in \mathcal A$.
Using the Rosenthal lemma (see e.g., \cite{diestel}) going to a subsequence, we may assume that 
$$\Sigma_{n\in\N\setminus\{k\}} |\mu_k(A_n)|<\varepsilon/3$$
holds for every $k\in \N$. It follows that $|\int 1_{\bigcup_{n\in \N}{[A_{2n}]}}dx_k^*|<\varepsilon/3$
if $k$ is an odd integer and $|\int 1_{\bigcup_{n\in \N}{[A_{2n}]}}dx_k^*|> 2\varepsilon/3$ if
$k$ is an even integer. In other words the element of the bidual
to $C(K_{\mathcal A})$ corresponding to the borel set $\bigcup_{n\in \N}{[A_{2n}]}$ of $K_{\mathcal A}$
witnesses the fact that $(x_n^*)_{n\in \N}$ is not weakly convergent. By the Grothendieck property
it is not weakly$^*$ convergent. As $(x_n^*)_{n\in \N}$ is bounded, it means that
$(\mu_n(A): n\in \N)$
is not a convergent sequence of the reals for some $A\in \mathcal A$ since
the span of characteristic functions of clopen sets is dense in $C(K_{\mathcal A})$ by
the Stone-Weierstrass theorem.

For the converse implication suppose that $(x_n^*)_{n\in\N}$ is a bounded sequence
in the dual to $C(K_{\mathcal A})$ which is not weakly convergent,  assume the condition of the lemma
 and let us conclude that  $(x_n^*)_{n\in\N}$ is not weakly$^*$ convergent.
$(x_n^*)_{n\in\N}$ as a nonconvergent sequence in a compact dual ball has at least two distinct
accumulation points. If every sequence of $(x_n^*)_{n\in\N}$ contains a weakly convergent
subsequence (and so weakly$^*$ convergent), then we conclude that $(x_n^*)_{n\in\N}$ is
not weakly$^*$ convergent as required.
Otherwise by choosing a subsequence of
$(x_n^*)_{n\in\N}$ we may assume that it does not have any weakly convergent subsequence.
So by the Eberlein-Smulian theorem $\{x_n^*: n\in \N\}$ is not relatively weakly compact, and hence
by the Grothendieck-Dieudonne characterization of weakly compact subsets of the duals to
Banach spaces $C(K)$ we obtain an antichain $\{A_n: n\in \N\}$ of $\mathcal A$, an
 $\varepsilon>0$, and an infinite $M\subseteq \N$ such that
such that $|\nu_n([A_n])|>\varepsilon$ where $\nu_n$s are the Radon measures
on $K_{\mathcal A}$ corresponding to $x_n^*$s.  Note that the restrictions $\mu_n$ of
$\nu_n$s to the family of all characteristic functions of clopen subsets of $K_{\mathcal A}$ satisfy
the condition of the  lemma, hence there is an $A\in \mathcal A$ as stated there. 
The function $1_{[A]}$  witnesses the fact that
$(x_n^*)_{n\in\N}$ is not weak$^*$ convergent as required.

\end{proof}

It is proved by Haydon  in 6.3 of \cite{haydonserdica} that Boolean algebras with the subsequential
separation property have the Grothendieck property.

\begin{definition}\label{positivegrothendef} We say that a Boolean algebra
 $\mathcal A$  a positive Grothendieck property if and only if
whenever 
\begin{itemize} 
\item $\{A_n: n\in \N\}$ is an antichain of $\mathcal A$,
\item $\varepsilon>0$,
\item $\mu_n$ is a bounded sequence of bounded, finitely additive non-negative measures on $\mathcal A$
such that $\mu_n(A_n)>\varepsilon$
\end{itemize}
then there is $A\in \mathcal A$ such that 
$$(\mu_n(A): n\in \N)$$
is not a convergent sequence of the reals.
\end{definition}

\begin{proposition} If a Boolean algebra $\mathcal A$ has a weakly subsequential separation
property, then $C(K_{\mathcal A})$ has the positive Grothendieck property.
\end{proposition}
\begin{proof} 
 Let $\{A_n: n\in \N\}$, $\varepsilon>0$
and $\mu_n$s be as in \ref{positivegrothendef}. 
Applying Lemmas 1 and  2 of \cite{talagrandnew} (see also 6.2., 6.3. of \cite{haydonserdica})
we may assume that there is an antichain $\{B_n: n\in \N\}$ and finitely additive
bounded measures $\lambda_n$ and $\nu_n$ such that $\mu_n=\lambda_n+\nu_n$,
where $\nu_n$s weakly converges, $\lambda_n(K)=\eta$ and 
$\lambda_n(B_n)>3\eta/4$.  By the Dieudonne-Grothendieck theorem 
applied to the $\mu_n$s we conclude that they do not form a relatively weakly
compact sets and hence are not a weakly convergent sequence, hence $\eta\not=0$.

Now use  the weak subsequential separation property to obtain $A\in \mathcal A$
such that both of the sets $M_1=\{n\in \N: B_n\leq A\}$ and $M_0=\{n\in \N: B_n\wedge A=0\}$
are infinite. For each $n\in M_1$ we have
 $$\mu_n(A)=\lambda_n(A\cap B_n)+\lambda_n(A\setminus B_n)+\nu_n(A)>3\eta/4-\eta/4+\nu_n(A)
=\eta/2+\nu_n(A),$$
and for each $n\in M_0$ we have
 $$\mu_n(A)=\lambda_n(A\cap B_n)+\lambda_n(A\setminus B_n)+\nu_n(A)<\eta/4+\nu_n(A).$$
As $\nu_n(A)$ converges, $\eta\not=0$, since weakly
 convergent sequences are weakly$^*$ convergent, we conclude that
$\mu_n(A)$s do not converge as required.

\end{proof}

\begin{proposition} There is a Boolean algebra with the weak subsequential
separation property which does not have the Grothendieck property and so does not have the
subsequential separation property.
\end{proposition}

\begin{proof} This is a classical example $\mathcal A$ (see e.g. \cite{schachermayer}) of the Boolean algebra
of all subsets $M$ of $\N$ such that $2k\in M$ if and only if $2k+1\in M$
for all but finitely many $k\in \N$. It is well known that
$\mu_n=\delta_{2n}-\delta_{2n+1}$ form a weakly$^*$ convergent sequence 
in $C(K_{\mathcal A})$ which is not
weakly convergent and so $\mathcal A$ does not have the Grothendieck property.
On the other hand given an antichain $\{A_n: n\in \N\}$ in $\mathcal A$ there is an infinite
$M\subseteq \N$ such that there are pairwise disjoint  $B_n\in \mathcal A$ with $A_n\subseteq B_n$
for $n\in M$ such that $2k\in B_n$ if and only if $2k+1\in B_n$ for all $k\in \N$ and al $n\in M$. 
Infinite unions
of such $B_n$s provide elements witnessing the separation.
\end{proof}

The Grothendieck property of $C(K_{\mathcal A})$ does not imply in ZFC the existence of an independent family
of cardinality continuum in $\mathcal A$. Namely assuming the continuum hypothesis Talagrand proved in 
\cite{talagrandnew} that there is a Boolean algebra $\mathcal A$ such that
$C(K_{\mathcal A})$ has the Grothendieck property but $l_\infty$ is not a quotient of
$C(K_{\mathcal A})$, in particular $\beta\N$ is not a subset of $K_{\mathcal A}$ and so
$\mathcal A$ has no uncountable independent family. Moreover
it is proved in \cite{brech} that
it is consistent that there is a Boolean algebra $\mathcal A$ which has the Grothendieck
property but has cardinality strictly smaller than $2^\omega$ (ground model after adding Sacks reals).
On the other hand 
 assuming $\frak p=2^\omega$ Haydon, Levy and Odell proved in \cite{HLO}
that each nonreflexive Banach space (in particular each of the form $C(K_{\mathcal A})$ 
for $\mathcal A$ infinite) with the
Grothendieck property has $l_\infty$ as a quotient. However we do not know the answer
to the following:

\begin{question}
 Is it consistent that  each Boolean algebra with the Grothendieck property 
has an independent family of cardinality continuum?
\end{question}

\section{Efimov's problem}

The affirmative answer to the question above could  be considered a weak solution to the Efimov problem
(see \cite{hart}) whether it is consistent if any compact $K$ without a nontrivial convergent sequence
has a subspace homeomorphic to $\beta\N$ or there is in ZFC a
compact space without a nontrivial convergent sequence and without a copy of $\beta\N$. 
Indeed, subsets of a compact space $K$ can be considered as subsets of the dual ball
to the Banach space $C(K)$ or Radon measures (points of $K$ correspond to pointwise measures)
with the weak$^*$ topology.
And so, instead of nontrivial convergent sequences
or copies of $\beta\N$ in $K$ one can ask for the same subspaces in the dual ball or the Radon measures.
The Grothendieck property 
of $C(K)$, in a sense, asserts that there are no nontrivial convergent sequences among Radon measures (not just pointwise measures, here nontrivial means
those which are not  convergent in the weak topology, or those which can be separated by a borel subset of 
the compact space, the notion changes as the dual ball always contains copies of intervals) and it easily implies 
the nonexistence of nontrivial (in the sense of having distinct terms) sequences of points of $K$. Of course
the negative answer to the above question would solve the original Efimov's problem. 

Note that Efimov's problem is equivalent to asking if an analogous property to our weak subsequential
separation property for points of the Stone space instead of elements of the Boolean algebra implies
the existence of an independent family of cardinality continuum.

Also Talagrand  proved (see \cite{talagrandnew}) that the dual ball to $C(K)$ contains a copy of $\beta\N$ if and only if
$\ell_\infty$ is a quotient of $C(K)$. So another  weak version of the Efimov problem would be to ask if it is consistent that whenever $K$ has no convergent sequence, then $\ell_\infty$ is a quotient of $C(K)$.
One should note here that consistently it is not the case (again the example of \cite{talagrandnew}) and that
the result of \cite{HLO} gives the consistency of: for every compact $K$ the dual ball to $C(K)$
with the weak$^*$ topology either contains a copy of $\beta\N$ or a convergent sequence which is not
weakly convergent.

\bibliographystyle{amsplain}

\vskip 15pt
\end{document}